\documentclass[12pt]{article}
\ifdefined\XeTeXrevision\else\ifdefined\pdfoutput\pdfoutput=1\fi\fi

\usepackage[english]{babel}
\usepackage[T1]{fontenc}
\usepackage{lmodern}
\usepackage{amsmath,amssymb,amsthm}
\usepackage[margin=1in]{geometry}
\usepackage{microtype}
\usepackage[hidelinks,hyperfootnotes=false]{hyperref}

\allowdisplaybreaks

\newtheoremstyle{plainsl}
  {\topsep}
  {\topsep}
  {\slshape}
  {}
  {\normalfont\bfseries}
  {.}
  { }
  {}

\theoremstyle{plainsl}
\newtheorem{theorem}{Theorem}[section]
\newtheorem{lemma}[theorem]{Lemma}
\newtheorem{corollary}[theorem]{Corollary}
\newtheorem*{conjecture}{Conjecture}

\theoremstyle{remark}
\newtheorem*{remark}{Remark}

\title{A dense-case theorem for Seymour's second neighborhood conjecture}
\author{Jake Brukhman\\{\small CoinFund}\\{\small\texttt{jake@coinfund.io}}}
\date{9 August 2026}

\hypersetup{
  pdftitle={A dense-case theorem for Seymour's second neighborhood conjecture},
  pdfauthor={Jake Brukhman}
}

\begin{document}
\maketitle
{\let\thefootnote\relax
}

\begin{abstract}
Seymour's second neighborhood conjecture asserts that every finite oriented
graph has a vertex with at least as many exact second outneighbors as
outneighbors.  Established cases include tournaments, proved by Fisher (1996),
and oriented graphs of minimum outdegree at most six, proved by Kaneko and
Locke (2001); a recent preprint of Sadhukhan, Sandeep, and Sen (2026) treats
minimum outdegree seven.  For dense incomplete graphs, Fidler and Yuster (2007)
proved the conjecture when the missing edges form a matching, a star, or a
clique, Ghazal (2012) extended this direction to generalized stars, and Dara,
Francis, Jacob, and Narayanan (2022) proved it when the missing edges can be
partitioned into a matching and a star.  We give a
short counting proof of the conjecture for every oriented graph of order
$n=2\delta+2$, where $\delta$ is the minimum outdegree, with no prescribed
structure on the missing edges.  Together with Fisher's tournament theorem,
this implies the conjecture for every oriented graph satisfying
$n\le2\delta+2$.  Combined with the known minimum-outdegree results, this
raises the best lower bound known to us on the order of a counterexample from
$16$ to $17$ and, conditional on the preprint of Sadhukhan, Sandeep, and Sen
(2026), from $18$ to $19$.
\end{abstract}

\section{Introduction}

A finite \emph{directed graph} is a pair
$G=(V(G),A(G))$, where $V(G)$ is a finite set of vertices and
$A(G)$ is a set of ordered pairs of distinct vertices, called \emph{arcs}.
We write
\[
 u\to v
\]
when $(u,v)\in A(G)$.  Thus the arrow points from the source of the arc to
its target.  The directed graph is \emph{oriented} if, for every two distinct
vertices $u,v$, at most one of $u\to v$ and $v\to u$ is present. For a vertex $v\in V(G)$, define its outdegree and indegree by
\[
 d^+(v)=|\{x\in V(G):v\to x\}|,
 \qquad
 d^-(v)=|\{x\in V(G):x\to v\}|.
\]
The minimum outdegree of $G$ is
$\delta=\min_{v\in V(G)}d^+(v).$ For a vertex $v\in V(G)$, define the \emph{outneighborhood} by
\[
 N_1(v)=\{x\in V(G):v\to x\},
\]
and the \emph{exact second outneighborhood} by
\[
 N_2(v)=
 \{x\in V(G)\setminus N_1(v):
   y\to x\text{ for some }y\in N_1(v)\}.
\]
By definition $N_2(v)$ is disjoint from $N_1(v)$; since $G$ is oriented, it also excludes $v$. Thus $N_2(v)$ consists exactly of the vertices reached from $v$ in two steps but not in one step.  When discussing reachability from $v$ to $x$, we call $v$ the \emph{source vertex} and $x$ the \emph{target vertex}. A vertex $v$ is a \emph{Seymour vertex} if $|N_2(v)|\ge|N_1(v)|.$

\begin{conjecture}[Seymour]
Every finite oriented graph has a Seymour vertex.
\end{conjecture}

Several results bear directly on dense or extremal cases.  Fidler and
Yuster~\cite{FidlerYuster2007} proved the conjecture for several prescribed
missing-edge structures, and Ghazal~\cite{Ghazal2012} generalized that line of
work.  Dara, Francis, Jacob, and
Narayanan~\cite{DaraFrancisJacobNarayanan2022} proved the conjecture when the
missing edges can be partitioned into a matching and a star.  Further
prescribed missing structures, including two stars and disjoint paths, were
treated by Daamouch, Ghazal, and
Al-Mniny~\cite{DaamouchGhazalAlMniny2024}, and Daamouch~\cite{Daamouch2021}
proved the conjecture for related digraph classes.  Espuny D\'iaz, Gir\~ao,
Granet, and Kronenberg~\cite{EspunyDiazEtAl2024} obtained reductions involving
minimum outdegree, and Guo, Kang, and Zwaneveld~\cite{GuoKangZwaneveld2026}
studied Seymour-tight orientations and counterexamples close to regular
tournaments.

Let $G$ be an oriented graph on $n$ vertices with minimum outdegree $\delta$. Counting arcs by their sources gives
\[
 n\delta\le|A(G)|\le\binom n2,
\]
so $n\ge2\delta+1$. Equality forces every pair of vertices to be adjacent and every outdegree to equal $\delta$, so the graph is a regular tournament, which satisfies the conjecture by Fisher's theorem~\cite{Fisher1996}.

\begin{theorem}[Dense-case theorem]\label{thm:fixed-target}
Let $G$ be an oriented graph on $n$ vertices with minimum outdegree $\delta$.
If
\[
 n=2\delta+2,
\]
then $G$ has a Seymour vertex.
\end{theorem}

At $n=2\delta+2$, up to $n/2$ edges may be missing and no structure is
imposed on them, in contrast
to~\cite{Daamouch2021,DaamouchGhazalAlMniny2024,DaraFrancisJacobNarayanan2022,FidlerYuster2007,Ghazal2012};
conversely, those results impose no lower bound on the minimum outdegree, so
neither result contains the other.

Corollary~\ref{cor:dense-case} shows that any counterexample to the
conjecture must satisfy $n\ge2\delta+3$: no counterexample can occur at the
two densest orders an oriented graph of minimum outdegree $\delta$ can have.
This complements the study of counterexamples close to regular tournaments
in~\cite{GuoKangZwaneveld2026}.

\begin{corollary}\label{cor:dense-case}
Let $G$ be an oriented graph on $n$ vertices with minimum outdegree $\delta$.
If $n\le2\delta+2$, then $G$ has a Seymour vertex.
\end{corollary}

\begin{proof}
Every oriented graph satisfies $n\ge2\delta+1$.  If $n=2\delta+1$, then $G$
is a regular tournament as observed above, and Fisher's
theorem~\cite{Fisher1996} applies.  If $n=2\delta+2$, apply
Theorem~\ref{thm:fixed-target}.
\end{proof}

\begin{corollary}[Minimum counterexample order]\label{cor:counterexample-order}
Every counterexample to Seymour's conjecture has at least $17$ vertices.
Assuming the result of Sadhukhan et al.~\cite{SadhukhanSandeepSen2026}, every
counterexample has at least $19$ vertices.
\end{corollary}

\begin{proof}
By Kaneko and Locke's theorem~\cite{KanekoLocke2001}, a counterexample has
minimum outdegree $\delta\ge7$.  Corollary~\ref{cor:dense-case} then gives
$n\ge2\delta+3\ge17$.  Under the result of Sadhukhan, Sandeep, and Sen, the
case $\delta=7$ is also excluded, so $\delta\ge8$ and hence $n\ge19$.
\end{proof}

Without Theorem~\ref{thm:fixed-target}, the same reasoning gives only
$n\ge2\delta+2$, since Fisher's theorem excludes just the order $n=2\delta+1$;
that is, $n\ge16$ unconditionally and $n\ge18$ conditionally.  To our
knowledge, no stronger order bounds were previously known.

\section{The fixed-target argument}

For a source vertex $v$, define
\[
 U(v)=V(G)\setminus\bigl(\{v\}\cup N_1(v)\cup N_2(v)\bigr).
\]
Thus $U(v)$ is the set of target vertices that the source $v$ cannot reach in
at most two steps. Now fix a target vertex $x$.  Define the \emph{trap} of $x$ by
\[
 T(x)=\{y\in V(G)\setminus\{x\}:y\nrightarrow x\}.
\]
These are exactly the vertices that do not send an arc to $x$.  Because $G$
is oriented, every outneighbor of $x$ belongs to $T(x)$.  Hence
\[
 |T(x)|\ge d^+(x)\ge\delta.
\]
Define the \emph{target slack}
\[
 q(x)=|T(x)|-\delta\ge0.
\]
Now define
\[
 P(x)=\{v\in V(G):x\in U(v)\}.
\]
Thus $P(x)$ is the set of source vertices from which the fixed target $x$
cannot be reached in at most two steps.  If $v\in P(x)$, then $v\nrightarrow
x$.  Moreover, no vertex of $N_1(v)$ points to $x$, since otherwise there
would be a two-step path from $v$ to $x$.  Therefore
\begin{equation}\label{eq:trap}
 \{v\}\cup N_1(v)\subseteq T(x).
\end{equation}
This is the fixed-target observation: every source $v$ that cannot reach $x$ in 2 steps, together with
all of its outneighbors, is confined to the same trap $T(x)$.

\begin{lemma}[Fixed-target capacity]\label{lem:capacity}
Let $x$ be a target vertex.  If $P(x)\ne\varnothing$, then
\[
 q(x)\ge1
 \qquad\text{and}\qquad
 |P(x)|\le2q(x)-1.
\]
In particular, if $q(x)=0$, then $P(x)=\varnothing$.
\end{lemma}

\begin{proof}
Suppose $P(x)\ne\varnothing$, and set $p:=|P(x)|\ge1$.
By \eqref{eq:trap}, $P(x)\subseteq T(x)$, and every arc whose source lies in
$P(x)$ has its target in $T(x)$.
There are at least $\delta p$ such arcs.  At most
$\binom{p}{2}$ have both endpoints in $P(x)$, because $G$ is oriented, and at
most $p\bigl(|T(x)|-p\bigr)$ go from $P(x)$ to the rest of the trap.
Consequently,
\[
 \delta p
 \le \binom{p}{2}
    +p\bigl(\delta+q(x)-p\bigr).
\]
Dividing by $p$ and rearranging gives $p\le2q(x)-1$.  Since $p\ge1$, this
forces $q(x)\ge1$.  The final assertion is the contrapositive.
\end{proof}

\begin{proof}[Proof of Theorem~\ref{thm:fixed-target}]
Suppose for contradiction that $G$ has no Seymour vertex.  Then $\delta>0$,
since a vertex of outdegree zero is automatically a Seymour vertex.  Let
\[
 L=\{v\in V(G):d^+(v)=\delta\},
 \qquad \ell=|L|\ge1.
\]
Every $v\in L$ is not a Seymour vertex, so $|N_2(v)|\le\delta-1$.
Since $n=2\delta+2$, it follows that
\begin{equation}\label{eq:demand}
 |U(v)|=n-1-\delta-|N_2(v)|\ge2.
\end{equation}
Thus each minimum-outdegree source contributes at least two unreachable
targets. Let
\[
 I=\bigl|\{(v,x)\in V(G)^2:x\in U(v)\}\bigr|.
\]
Counting these ordered source--target pairs first by their source and then by
their target gives
\begin{equation}\label{eq:doublecount}
 I=\sum_{v\in V(G)}|U(v)|
 =\sum_{x\in V(G)}|P(x)|
\end{equation}
since $x\in U(v)$ and $v\in P(x)$ describe the
same ordered pair $(v,x)$.  From \eqref{eq:demand},
\begin{equation}\label{eq:lower}
 I\ge2\ell>0.
\end{equation}

\noindent Put
\[
 Q=\sum_{x\in V(G)}q(x).
\]
Since $I>0$, at least one set $P(x)$ is nonempty; Lemma~\ref{lem:capacity}
then implies that at least one $q(x)$ is positive.  Applying the lemma
in \eqref{eq:doublecount} gives
\begin{equation}\label{eq:capacity}
 I\le\sum_{q(x)>0}\bigl(2q(x)-1\bigr)
   <2\sum_xq(x)=2Q.
\end{equation}
It remains to compare $Q$ with $\ell$.  Since
$|T(x)|=n-1-d^-(x)$,
\[
 q(x)=n-1-\delta-d^-(x).
\]
Total indegree equals total outdegree, and every vertex outside $L$ has
outdegree at least $\delta+1$.  Therefore
\begin{equation}\label{eq:ledger}
\begin{aligned}
 Q
 &=n(n-1-\delta)-\sum_xd^+(x)\\
 &\le n(n-1-\delta)
   -\bigl(\ell\delta+(n-\ell)(\delta+1)\bigr)\\
 &=\ell,
\end{aligned}
\end{equation}
where the last equality uses $n=2\delta+2$.
Now \eqref{eq:lower}--\eqref{eq:ledger} give
\[
 I<2Q\le2\ell\le I,
\]
a contradiction.  Hence $G$ has a Seymour vertex.
\end{proof}

\begin{remark}
The order $n=2\delta+2$ also represents the limitation of this argument.  For
$n=2\delta+3$, the same computation as in \eqref{eq:ledger} gives only
$Q\le\ell+n$, while \eqref{eq:demand} gives $I\ge3\ell$; since $\ell\le n$,
the comparison $I<2Q\le2\ell+2n$ no longer yields a contradiction.  Extending
the method past $n=2\delta+2$ therefore requires new ideas.
\end{remark}

\section*{Acknowledgments}

The author initiated and directed the investigation, curated intermediate
results, selected the theorem for publication, and edited the final statement
and exposition.  OpenAI language models (GPT-5 family) carried out the
detailed mathematical exploration, implemented counterexample searches and
verification tools, discovered the fixed-target capacity argument and its
double-counting proof, and drafted the manuscript; Anthropic Claude models
performed an adversarial audit of an intermediate draft and assisted with
revisions.  The author verified
the proofs and accepts sole responsibility for the final manuscript and its
claims.

\end{document}